\documentclass[pdflatex,sn-mathphys-num]{sn-jnl}
\usepackage{amsmath,mathrsfs,amsmath,amssymb,enumerate}
\usepackage{cleveref}
\usepackage{dsfont}
\usepackage{comment}
\usepackage{indentfirst}
\usepackage{url}

\usepackage{multirow}%
\usepackage{amsfonts}%
\usepackage{amsthm}%
\usepackage{mathrsfs}%

\usepackage{xcolor}%
\usepackage{textcomp}%
\usepackage{manyfoot}%

\numberwithin{equation}{section}

\theoremstyle{thmstyleone}
\newtheorem{thm}{Theorem}[section]

\newtheorem{prop}[thm]{Proposition}

\theoremstyle{thmstyletwo}

\newtheorem{rem}[thm]{\bf Remark}

\theoremstyle{thmstylethree}

\mathchardef\lz="2D

\newcommand{\R}{\mathbb{R}^N}

\newcommand{\RR}{\mathbb{R}}

\allowdisplaybreaks

\crefname{equation}{Problem}{Problem}
\crefname{thm}{Theorem}{Theorem}
\crefname{lem}{Lemma}{Lemma}
\crefname{prop}{Proposition}{Propositions}

\crefalias{prop}{prop}

\begin{document}

\title{The Li--Lin's open problem on $\mathbb{R}^N$}

\author[1]{\fnm{Zhi-Yun} \sur{Tang}}

\author[1]{\fnm{Xianhua} \sur{Tang}}

\abstract{
	In 2012, Y.Y. Li and C.-S. Lin (Arch. Ration. Mech. Anal., 203(3): 943–968) posed an open problem concerning the existence of positive solutions to the elliptic equation
	$$
	\begin{cases}
		-\Delta u = -\lambda |x|^{-s_1}|u|^{p-2}u + |x|^{-s_2}|u|^{q-2}u & \text{in } \Omega, \\
		u = 0 & \text{on } \partial \Omega,
	\end{cases}
	$$
	for $\lambda > 0$, $p > q = 2^*(s_2)$, $0 \leq s_1 < s_2 < 2$, and $2^*(s) = \frac{2(N-s)}{N-2}$ denotes the Hardy-Sobolev critical exponent, initially studied in bounded domains $\Omega \subset \mathbb{R}^N$, $N \geq 3$. Currently, research on this open problem remains limited, and a complete resolution is still far from being achieved.

	Motivated by the need to address this open problem in more general settings, we extend our investigation to the entire space $\mathbb{R}^N$, focusing on the equation
	$$
	-\Delta u + u = -\lambda |x|^{-s_1}|u|^{p-2}u + |x|^{-s_2}|u|^{q-2}u \quad \text{in } \mathbb{R}^N.
	$$

Our analysis reveals stark contrasts between bounded and unbounded domains: in $\mathbb{R}^N$, the equation admits no solution when $q = 2^*(s_2)$ for any $\lambda > 0$, whereas a positive solution exists when $q < 2^*(s_2)$.

To establish these results, we employ the Nehari manifold method; however, the functional's unboundedness from below on the manifold causes standard global minimization techniques to be inapplicable. Instead, we characterize a local minimizer of the energy functional on the Nehari manifold, overcoming the challenge posed by the lack of a global minimizer. }

\keywords{Positive solutions, Nehari manifold, Hardy-Sobolev critical exponents, Local minimum point}
\pacs[MSC Classification]{35D99; 35J15; 35J91}
\maketitle
\section{Introduction}
Consider the following nonlinear equation involving Hardy-Sobolev exponents:
\begin{equation}\label{001}
	\begin{cases}
		-\Delta u=-\lambda |x|^{-s_{1}}|u|^{p-2}u+|x|^{-s_{2}}|u|^{q-2}u &\text { in } \Omega, \\
		u(x)=0 &\text { on } \partial \Omega,
	\end{cases}
\end{equation}
where $\Delta$ denotes the standard Laplace operator, $\Omega$ is a smooth bounded domain in $\mathbb{R}^N$ ($N \geq 3$), $\lambda \in \mathbb{R}$, $0 \leq s_1 < s_2 < 2$, $2^*(s) = \frac{2(N-s)}{N-2}$ for $0 \leq s \leq 2$, $2 < p \leq 2^*(s_1)$, and $2 < q \leq 2^*(s_2)$.

H. Brézis and L. Nirenberg \cite{Brezis1983} studied this equation and obtained several existence results for positive solutions. In the case involving the Sobolev critical exponent, they set $s_1 = s_2 = 0$, $p = 2$, $q = 2^*(s_2) = 2^*(0) = \frac{2N}{N-2}$, and $\lambda \in (-\lambda_1, 0)$, where $\lambda_1$ denotes the first eigenvalue of the $-\Delta$ operator with zero Dirichlet boundary conditions.

Subsequently, E. Jannelli \cite{Jannelli1999} studied the case involving Hardy terms in 1999, and N. Ghoussoub and C. Yuan \cite{Ghoussoub2000} stuidied the case involving Hardy-Sobolev critical exponents in 2000, both generalizing the results of H. Brézis and L. Nirenberg from 1983. For studies on more general cases, see \cite{Ekeland2002,Chen2004,Smets2005,Han2007,Li2012,Deng2012,Zhang2016a,Li2022}.

In 2004, N. Ghoussoub and X.S. Kang \cite{Ghoussoub2004} considered the case where $0$ is located on $\partial\Omega$ instead of placing $0$ inside $\Omega$ in their study. For subsequent studies, one may refer to \cite{Ghoussoub2006,Ghoussoub2006a,Ghoussoub2009,Hsia2010,Ghoussoub2017,Wang2017}.

In the literature \cite{Li2012}, Y.Y. Li and C.-S. Lin investigated the case of equation \eqref{001} with dual Hardy-Sobolev critical exponents ($0 \in \partial\Omega$, $p = 2^*(s_1) < q = 2^*(s_2)$, $\lambda \in \mathbb{R}$). Additional related advancements can also be found in \cite{Jin2011,Yan2013,Cerami2015,Zhong2021,Wang2021,Wang2025,Tang2025} and the references therein.

Notably, most existing studies on equation \eqref{001} are primarily conducted under the conditions of $\lambda < 0$ or $p < q$, where it is relatively straightforward to prove the boundedness of $(PS)_c$ sequences and thereby establish the existence of positive solutions. However, when $\lambda > 0$ and $p > q$, $(PS)_c$ sequences may not be bounded, making it difficult to obtain satisfactory results.

Y.Y. Li and C.-S. Lin \cite{Li2012} identified this challenge in their research and posed the following open problem: When \( s_1 < s_2 \) and \( \lambda > 0 \), the existence of positive solutions to equation \eqref{001} remains completely unknown. Specifically, for the equation
$$
\Delta u - u^p + \frac{u^{2^*(s)-1}}{|x|^s} = 0 \quad \text{in } \Omega,
$$
where \( 0 < s < 2 \) and \( 2^*(s) - 1 < p < \frac{N+2}{N-2} \), the existence of positive solutions remains an interesting open problem.

Firstly, in 2015, G. Cerami, X. Zhong and W. Zou \cite{Cerami2015} obtained some existence results of positive solutions by the perturbation approach and the monotonicity trick. Recently in \cite{Tang2025}, we gave a first nonexistence result by contradiction with the H\"{o}lder inequality, the Hardy inequality and the Young inequality to the two-critical Li--Lin's open problem, and a second existence result to the Li--Lin's open problem ($2^{*}(s_1)\geq p>q=2^{*}(s_2)$) by the method of sub-supersolutions based on the Theorem 1.5 in \cite{Cerami2015}.

Initial studies of the above problem have focused on bounded domains. Motivated by resolving this open problem in more general settings, we aim to analyze how the equation’s solutions behave when the domain is an unbounded region, particularly $\mathbb{R}^N$.

Now, we consider the Hardy-Sobolev's critical exponents problem
\begin{equation}
	\label{eq1}
	-\Delta u+u=-\lambda |x|^{-s_{1}}|u|^{p-2}u+|x|^{-s_{2}}|u|^{q-2}u \text{ in } \mathbb{R}^N,
\end{equation}
where $N\geq3$, $\lambda\in \mathbb{R}$, $0\leq s_1 <s_2 < 2$, $2^{*}(s)=\frac{2(N-s)}{N-2}$ for $0 \leq s \leq 2$, $2< p\leq 2^{*}(s_1)$, $2< q\leq 2^{*}(s_2)$.

In fact, when studying \cref{eq1}, we'll find that the situation in the unbounded case is quite different from that in the bounded region. Specifically, in $\mathbb{R}^N$, for any $\lambda > 0$, the equation has no solution when $q = 2^{*}(s_2)$, while it has a solution when $q < 2^{*}(s_2)$.

To prove these results, we use the Nehari manifold method. However, it's important to note that for this type of problem, the functional is unbounded from below on the manifold, so the common Nehari manifold method doesn't work here. Therefore, we're trying to find a local minimum on the Nehari manifold to obtain a solution, and the main challenge now is how to find this local minimum.

We will present the main results regarding the existence and nonexistence of solutions of this question in the unbounded region $\mathbb{R}^N$. The main results are the following theorems.
\begin{thm}\label{thm1}
	Suppose that $0\leq s_1 < s_2 < 2$ and $2^{*}(s_2)=q<p\leq 2^{*}(s_1)$. Then \cref{eq1} has no nonzero solution for all $\lambda>0$.
\end{thm}
\begin{rem}
	Contrary to the situation in bounded domains, the Li--Lin problem has no solution in $\mathbb{R}^N$.
\end{rem}
\begin{thm}\label{thm2}
	Suppose that $0\leq s_1 < s_2 < 2$, $2<q<p< 2^{*}(s_1)$ and $q< 2^{*}(s_2)$. Assume that \begin{eqnarray}\label{21}
		q>\frac{2 - s_2}{2 - s_1}p+\frac{2s_2 - 2s_1}{2 - s_1}.
	\end{eqnarray} Then \cref{eq1} has at least one positive solution for all $\lambda>0$.
\end{thm}
\begin{rem}
	Since the functional fails to be bounded below on the manifold and does not attain a global minimum, we turn to identifying solutions through the discovery of local minima, differing from the standard approach.
\end{rem}
\begin{rem}
	Unlike bounded domains, distinct phenomena emerge in $\mathbb{R}^N$ for the critical and subcritical cases. Particularly, in $\mathbb{R}^N$, there exists no solution when $q = 2^{*}(s_2)$, while a positive solution is ensured to exist when $q < 2^{*}(s_2)$.
\end{rem}

\section{Preliminaries}
Our work space is the Sobolev's space $H^1{(\R)}$ with scalar product and norm given by
\begin{equation*}
(u,v)=\int_{\R }\nabla u\cdot\nabla vdx + \int_{\R } uvdx\quad \mbox{and}\quad {\|u\|}=(u,u)^{\frac{1}{2}}.
\end{equation*}
It is well-known that the solutions of problem \eqref{eq1} are precisely the critical points of the energy  functional $I$: $H^1{(\R  )}\rightarrow \mathbb{R}$ defined by
\begin{eqnarray*}
I(u)&=& \frac{1}{2}\int_{\R }\left| \nabla u\right|^2 dx + \frac{1}{2}\int_{\R } \left| u\right| ^2dx+\frac{\lambda}{p}\int_{\R }|x|^{-s_{1}}|u|^p d x-\frac{1}{q}\int_{\R }|x|^{-s_{2}}|u|^q d x.
\end{eqnarray*}
It is easy to see that $I \in C^2(H^1{(\R)},\mathbb{R})$. Moreover, for any $u,v$ and $w\in H^1{(\R)}$, we have
\begin{eqnarray*}
\langle I'(u),\ v\rangle &=&(u,v)+\lambda\int_{\R }|x|^{-s_{1}}|u|^{p-2}uv dx- \int_{\R }|x|^{-s_{2}}|u|^{q-2} uv d x
\end{eqnarray*}
and
\begin{eqnarray*}
\langle I''(u)v,\ w\rangle =(v,w)+(p-1)\lambda\int_{\R }|x|^{-s_{1}}|u|^{p-2}vw dx- (q-1)\int_{\R }|x|^{-s_{2}}|u|^{q-2} vw d x.
\end{eqnarray*}
Define
\begin{eqnarray*}
	M&=&\left\{u\in H_r^1(\R )\setminus \{0\}\left| \varphi(u)\stackrel{\triangle}{=}\langle I'(u),u\rangle=0\right.\right\},\\
	M^+&=&\left\{u\in M\left| \psi(u)\stackrel{\triangle}{=}\langle I''(u)u,\ u\rangle<0\right.\right\},\\ m^+&=&\inf\{I(u)\mid u\in M^+\},
\end{eqnarray*}
and
\begin{eqnarray*}
	M^0&=&\left\{u\in M\left| \psi(u)=0 \right.  \right\},\\
	m^0&=&\inf\{I(u)\mid u\in M^0\},
\end{eqnarray*}
where $$H_r^1(\R )=\{ u \in  H^1(\R )\mid u\text{ is radially symmetric}\}.$$  To prove our theorems, we firstly give some propositions.

\begin{prop}\label{prop:1}
	Suppose that $0\leq s_1 < s_2 < 2$, $2<q<p< 2^{*}(s_1)$ and $q< 2^{*}(s_2)$. Assume that $\varphi(u_0)<0$ for some $u_0$. Then there exists $t_0\in (0, 1)$ such that $t_0 u_0\in M^+$, thus $M^+ \not= \phi$.
\end{prop}
\begin{proof}
It follows from $\varphi(u_0)<0$ and $\varphi(0)=0$ that $u_0\neq0$. Let
\begin{eqnarray*}
	g(t)&=&t^{-2}\varphi(tu_0)\\&=&\|u_0\|^2+\lambda t^{p-2}\int_{\R }|x|^{-s_{1}}|u_0|^p d x-t^{q-2}\int_{\R }|x|^{-s_{2}}|u_0|^q d x
\end{eqnarray*}
for $t>0$. Then we have
\begin{eqnarray*}
	g(1)&=&\varphi(u_0)<0,\\
	g(+0)&=&\|u_0\|^2>0,
\end{eqnarray*}
and $$g(+\infty)=+\infty.$$ Hence, there exist $t_0\in (0, 1)$ and $t_1\in (1, +\infty)$ such that $$g(t_0)=g(t_1)=0,$$ which implies that
\begin{eqnarray*}
	\lambda \int_{\R }|x|^{-s_{1}}|u_0|^p d x=\frac{t_1^{q-2}-t_0^{q-2}}{t_1^{p-2}-t_0^{p-2}}\int_{\R }|x|^{-s_{2}}|u_0|^q d x.
\end{eqnarray*}
Moreover, one has
$$
\varphi(t_0 u_0)=t_0^2g(t_0)=0.
$$

Let $$h(x)=(p-2)(x^{q-2}-1)-(q-2)(x^{p-2}-1)$$ for $x>0$. When $x>1$,  $$h'(x)=(p-2)(q-2)(x^{q-3}-x^{p-3})<0.$$ Since $h(x)$ is decreasing for $x > 1$, it follows that $$h(x) < h(1) = 0\quad \text{for}\ x>1.$$ Furthermore, we obtain
\begin{eqnarray*}
	\psi(t_0 u_0)&=&t_0^2\|u_0\|^2+(p-1)t_0^p \lambda \int_{\R }|x|^{-s_{1}}|u_0|^p d x\\&&-(q-1)t_0^q\int_{\R }|x|^{-s_{2}}|u_0|^q d x-\varphi(t_0 u_0)\\
	&=&(p-2)t_0^p \lambda \int_{\R }|x|^{-s_{1}}|u_0|^p d x-(q-2)t_0^q\int_{\R }|x|^{-s_{2}}|u_0|^q d x\\
	&=&(p-2)t_0^p\frac{t_1^{q-2}-t_0^{q-2}}{t_1^{p-2}-t_0^{p-2}}\int_{\R }|x|^{-s_{2}}|u_0|^q d x-(q-2)t_0^q\int_{\R }|x|^{-s_{2}}|u_0|^q d x\\
	&=&[(p-2)t_0^{p-q}(t_1^{q-2}-t_0^{q-2})-(q-2)(t_1^{p-2}-t_0^{p-2})]	\frac{t_0^q\int_{\R }|x|^{-s_{2}}|u_0|^q d x}{t_1^{p-2}-t_0^{p-2}}\\
	&=&h\left( \frac{t_1}{t_0}\right) \frac{t_0^{p+q-2}\int_{\R }|x|^{-s_{2}}|u_0|^q d x}{t_1^{p-2}-t_0^{p-2}}\\
	&<&0.
\end{eqnarray*}
Thus, $t_0 u_0\in M^+$, and $M^+ \not= \phi$.
\end{proof}

\begin{prop}\label{prop:2}
	Suppose that $0\leq s_1 < s_2 < 2$, $2<q<p< 2^{*}(s_1)$ and $q< 2^{*}(s_2)$. Assume that \eqref{21} holds. Then $M^0 \not= \phi$, and $M^+ \not= \phi$.
\end{prop}
\begin{proof}
	Given \(u\in H_r^1(\R)\setminus \{0\}\), for \(r > 0\) and \(x\in \R\), we define $$u_{r}(x)=u(r^{-1}x).$$ It's easy to know $u_r\in H^1_r(\R  )$.
	Define
	\begin{eqnarray*}
		g(r)&\stackrel{\triangle}{=}&\int_{\R }\left| \nabla u\right|^2 dx+r^{2}\int_{\R } \left| u\right| ^2dx\\&&-\frac{p-q}{p-2}\int_{\R }|x|^{-s_{2}}|u|^q d x\left(\frac{(q-2)\int_{\R }|x|^{-s_{2}}|u|^q d x}{\lambda(p-2)\int_{\R }|x|^{-s_{1}}|u|^p d x}\right)^\frac{q-2}{p-q}r^{2-s_2-(q-2)\frac{s_2 - s_1}{p - q}}.
	\end{eqnarray*}
	By \eqref{21}, we have $$2-s_2-(q-2)\frac{s_2 - s_1}{p - q}<0.$$ Hence, $g$ is strictly increasing on (0,+$\infty$). Moreover,
	\begin{eqnarray*}
		g(+0)&=&-\infty, \\
		g(+\infty)&=&+\infty,
	\end{eqnarray*} Thus, there exists unique $r_0\in $ $(0,+\infty$) such that $g(r_0)=0$. Let
	\begin{eqnarray*}
		t_0=\left(\frac{(q-2)\int_{\R }|x|^{-s_{2}}|u|^q d x}{\lambda(p-2)\int_{\R }|x|^{-s_{1}}|u|^p d x}\right)^\frac{1}{p-q}r_0^{-\frac{s_2 - s_1}{p - q}}.
	\end{eqnarray*}
	Then we have
	\begin{eqnarray*}
		\varphi(t_0u_{r_0})&=&t_0^2r_0^{N-2}\int_{\R }\left| \nabla u\right|^2 dx+t_0^2r_0^{N}\int_{\R } \left| u\right| ^2dx\\&&+ t_0^pr_0^{N-s_1}\lambda\int_{\R }|x|^{-s_{1}}|u|^p d x-t_0^qr_0^{N-s_2}\int_{\R }|x|^{-s_{2}}|u|^q d x\\&=&t_0^2r_0^{N-2}g(r_0)\\&=&0,
	\end{eqnarray*}
	and
	\begin{eqnarray*}
		\psi(t_0u_{r_0})&=&t_0^2r_0^{N-2}\int_{\R }\left| \nabla u\right|^2 dx+t_0^2r_0^{N}\int_{\R } \left| u\right| ^2dx\\&&+ (p-1)t_0^pr_0^{N-s_1}\lambda\int_{\R }|x|^{-s_{1}}|u|^p d x \\&&-(q-1)t_0^qr_0^{N-s_2}\int_{\R }|x|^{-s_{2}}|u|^q d x-\varphi(t_0u_{r_0})\\
		&=& (p-2)t_0^pr_0^{N-s_1}\lambda\int_{\R }|x|^{-s_{1}}|u|^p d x-(q-2)t_0^qr_0^{N-s_2}\int_{\R }|x|^{-s_{2}}|u|^q d x\\&=&0,
	\end{eqnarray*}
	which implies that $t_0u_{r_0}\in M^0$. Thus, $M^0 \not= \phi$.

	Assume that $u\in M^0$. For $r>0$, let
	\begin{eqnarray*}
		h(r)&=&\int_{\R }\left| \nabla u\right|^2 dx+r^2\int_{\R } \left| u\right| ^2dx\\&&+ r^{2-s_1}\lambda\int_{\R }|x|^{-s_{1}}|u|^p d x-r^{2-s_2}\int_{\R }|x|^{-s_{2}}|u|^q d x.
	\end{eqnarray*}
	Then $h(1)=0$ and
	\begin{eqnarray*}
		h'(1)&=&2\int_{\R } \left| u\right| ^2dx+ (2-s_1)\lambda\int_{\R }|x|^{-s_{1}}|u|^p d x-(2-s_2)\int_{\R }|x|^{-s_{2}}|u|^q d x\\&=&2\int_{\R } \left| u\right| ^2dx+ (2-s_1)\frac{q-2}{p-q}\|u\|^2-(2-s_2)\frac{p-2}{p-q}\|u\|^2\\&>&0.
	\end{eqnarray*} Hence, there exists $\delta>0$ such that	$h(r)<0$ for $r\in (1-\delta, 1)$. Thus $$\varphi(u_r)=r^{N-2}h(r)<0\quad \text{for} \ r\in (1-\delta, 1).$$ By \cref{prop:1}, we have $M^+ \not= \phi$.
\end{proof}

\begin{prop}\label{m^+}
	Suppose that $0\leq s_1 < s_2 < 2$, $2<q<p< 2^{*}(s_1)$ and $q< 2^{*}(s_2)$. Assume that \eqref{21} holds. Then  $0<m^0<+\infty$, and $0<m^+<+\infty$.
\end{prop}
\begin{proof}
	From \cref{prop:2}, we know that $M_0 \neq \emptyset$, and furthermore, one has $m^0 < +\infty$. For any $u \in M^0$, we have $\varphi(u) = 0$ and $\psi(u) = 0$,
	which implies that
	\begin{eqnarray*}
		\lambda\int_{\R }|x|^{-s_{1}}|u|^p d x=\frac{q-2}{p-q}\|u\|^2,\\
		\int_{\R }|x|^{-s_{2}}|u|^q d x=\frac{p-2}{p-q}\|u\|^2.
	\end{eqnarray*}
	It follows from $\varphi(u)=0$ and Hardy-Sobolev inequality that
	\begin{eqnarray*}
		\|u\|^2&\leq&\|u\|^2+\lambda\int_{\R }|x|^{-s_{1}}|u|^p d x\\
		&=&\int_{\R }|x|^{-s_{2}}|u|^q d x	\\
		&\leq& C\|u\|^q
	\end{eqnarray*}
	for all $u\in M$ and some constant $C>0$, which implies that $\|u\|\geq C^{-\frac{1}{q-2}}$. Hence, one has
	\begin{eqnarray*}
		I(u)&=&\frac{(p-2)(q-2)}{2pq}\|u\|^2\\
		&\geq& \frac{(p-2)(q-2)}{2pq}C^{-\frac{2}{q-2}}\\
		&>&0
	\end{eqnarray*}
	for all $u \in M^0$. Then $m^0>0.$ Similarly, we have $0<m^+<+\infty$.
\end{proof}

\begin{prop}\label{m^+<m^0}
Suppose that $0\leq s_1 < s_2 < 2$, $2<q<p<2^{*}(s_1)$ and $q< 2^{*}(s_2)$. Assume that \eqref{21} holds and there exists an $u_0\in M^0$ such that $I(u_0)=m^0$. Then we have $m^+<m^0$.
\end{prop}
\begin{proof}
It follows from $u_0\in M^0$ that $\varphi'(u_0)\not=0$. If not, $\varphi'(u_0)=0$, $u_0$ is a solution of the following equation
\begin{equation*}
	-2\Delta u+2u=-\lambda p |x|^{-s_{1}}|u|^{p-2}u+q|x|^{-s_{2}}|u|^{q-2}u \text {\ \ in } \R.
\end{equation*}
Hence, $u_0$ satisfies the $Pohozave$ identity
\begin{eqnarray*}
	0&=&(N-2)\int_{\R }\left| \nabla u_0\right|^2 dx+N\int_{\R } \left| u_0\right| ^2dx\\&&+\lambda (N-s_1)\int_{\R }|x|^{-s_{1}}|u_0|^p d x-(N-s_2)\int_{\R }|x|^{-s_{2}}|u_0|^q d x.
\end{eqnarray*}
and which implies that
\begin{eqnarray}\label{232}
	0&=&2\int_{\R } \left| u_0\right| ^2dx+\lambda(2-s_1)\int_{\R }|x|^{-s_{1}}|u_0|^p d x \notag\\
	&&- (2-s_2)\int_{\R }|x|^{-s_{2}}|u_0|^q d x
\end{eqnarray}
by $\varphi(u_0)=0$. Moreover, from $\varphi(u_0)=0$ and $\psi(u_0)=0$, we obtain
\begin{eqnarray}\label{233}
	(p-2)\lambda\int_{\R }|x|^{-s_{1}}|u_0|^p d x-(q-2)\int_{\R }|x|^{-s_{2}}|u_0|^q d x=0.
\end{eqnarray}
By (\ref{232}) and (\ref{233}), we have
\begin{eqnarray*}
	2(q-2)\int_{\R } \left| u_0\right| ^2dx+\lambda[(2-s_1)(q-2)-(2-s_2)(p-2)]\int_{\R }|x|^{-s_{1}}|u_0|^p d x=0
\end{eqnarray*}
which implies that $u_0=0$, due to \eqref{21}. But $u_0\not=0$ because of $u_0\in M^0$. Hence, $\varphi'(u_0)\not=0$.

For $t>0$ and $r>0$, let
\begin{eqnarray*}
	h(t,r)&=&\varphi (t(u_0+rv_0)),
\end{eqnarray*}
where $v_0\in H^1_r(\R)$ satisfying $$\langle\varphi'(u_0),v_0\rangle>0$$ by  $\varphi'(u_0)\not=0$.
Then one has
\begin{eqnarray*}
	h_t(t,r)&=&\frac{\partial h}{\partial t}(t,r)\\&=&\langle \varphi  '(t(u_0+rv_0)),u_0+rv_0\rangle,\\
	h_{tt}(t,r)&=&\langle \varphi  ''(t(u_0+rv_0))(u_0+rv_0),u_0+rv_0\rangle
\end{eqnarray*}
and
\begin{eqnarray*}
	h_r(t,r)&=&\langle \varphi  '(t(u_0+rv_0)),tv_0\rangle.
\end{eqnarray*}
Then we have
\begin{eqnarray*}
	h_t(1,0)&=&\langle \varphi  '(u_0),u_0\rangle=0,\\
	h_{tt}(1,0)&=&\langle \varphi  ''(u_0)u_0,u_0\rangle>0
\end{eqnarray*}
and
\begin{eqnarray*}
	h_r(1,0)&=&\langle \varphi  '(u_0),v_0\rangle>0.
\end{eqnarray*}
Due to
\begin{eqnarray*}
	h(1,0)&=&0,\\
	h_r(1,0)&=&\langle \varphi '(u_0),v_0\rangle>0,
\end{eqnarray*}
and $$h\in C^2(\RR^+\times \RR).$$
It follows from the implicit function theorem that there exists a constant $\delta>0$ and a $C^2$ function $r=r(t)$ from $(1-\delta, 1+\delta)$ to $(-\delta, \delta)$ such that $r(1)=0$ and $$h(t,r(t))=0.$$ Hence, we have $$t(u_0+r(t)v_0)\in M \quad \text{for}\ t\in (1-\delta, 1+\delta).$$ Furthermore, one has
\begin{eqnarray*}
	r'(1)&=&-\frac{h_t}{h_r}=0,
\end{eqnarray*}
and
\begin{eqnarray*}
	r''(1)&=&\frac{2h_th_rh_{tr}-h_r^2h_{tt}-h_t^2h_{rr}}{h_r^3}\\
	&=&\frac{-h_{tt}}{h_r}\\
	&=&\frac{-\langle \varphi ''(u_0)u_0,u_0\rangle}{\langle \varphi '(u_0),v_0\rangle}\\
	&<&0.
\end{eqnarray*}
Define
$$
f(t)=\psi(t(u_0+r(t)v_0)) \quad \text{for}\ t>0.
$$
Then one has
\begin{eqnarray*}
	f'(t)&=&\langle\psi'(t(u_0+r(t)v_0)),u_0+r(t)v_0+tr'(t)v_0\rangle.
\end{eqnarray*}
Hence, we have $f(1)=0$ and
\begin{eqnarray*}
	f'(1)&=&\langle\psi'(u_0),u_0\rangle\\
	&=&2\|u_0\|^2+p(p-1)\lambda\int_{\R }|x|^{-s_{1}}|u_0|^pdx-q(q-1)\int_{\R }|x|^{-s_{2}}|u_0|^q d x\\
	&=&2\|u_0\|^2+p(p-1)\frac{q-2}{p-q}\|u_0\|^2-q(q-1)\frac{p-2}{p-q}\|u_0\|^2\\
	&=&(p-2)(q-2)\|u_0\|^2\\
	&>&0.
\end{eqnarray*}
There exists $\delta_1\in(0,\delta)$ such that $$t(u_0+r(t)v_0)\in M^+,$$ because that $$\psi(t(u_0+r(t)v_0))=f(t)<0$$ for $t\in (1-\delta_1, 1)$.

Moreover, define
$$
g(t)=I(t(u_0+r(t)v_0)) \quad \text{for}\ t>0.
$$
Then one has
\begin{eqnarray*}
	g'(t)&=&\langle I'(t(u_0+r(t)v_0)),u_0+r(t)v_0+tr'(t)v_0\rangle\\
	g''(t)&=&\langle I''(t(u_0+r(t)v_0))(u_0+r(t)v_0+tr'(t)v_0),u_0+r(t)v_0+tr'(t)v_0\rangle\\
	&&+\langle I'(t(u_0+r(t)v_0)),2r'(t)v_0+tr''(t)v_0\rangle.
\end{eqnarray*}
We have
\begin{eqnarray*}
	g'(1)&=&\langle I'(u_0),u_0\rangle\\
	&=&\varphi(u_0)\\
	&=&0,\\
	g''(1)&=&\langle I''(u_0)u_0,u_0\rangle+r''(1)\langle I'(u_0),v_0\rangle\\&=&\psi(u_0)-\langle \varphi''(u_0)u_0,u_0\rangle\\
	&=&-\langle \varphi''(u_0)u_0,u_0\rangle\\
	&=&-2\|u_0\|^2-p(p-1)\lambda\int_{\R }|x|^{-s_{1}}|u_0|^{p} dx+q(q-1)\int_{\R }|x|^{-s_{2}}|u_0|^q d x\\
	&=&-2\|u_0\|^2-p(p-1)\frac{q-2}{p-q}\|u_0\|^2+q(q-1)\frac{p-2}{p-q}\|u_0\|^2\\
	&=&-(p-2)(q-2)\|u_0\|^2\\
	&<&0.
\end{eqnarray*}
There exists $\delta_2\in(0,\delta_1)$ such that $$m^+\leq I(t(u_0+r(t)v_0))=g(t)<g(1)=I(u_0)=m^0$$ for $t\in (1-\delta_2, 1).$ Then we complete our proof.
\end{proof}

\section{Proof of Theorems \ref{thm1} and \ref{thm2} }
We now proceed to prove the main Theorems.

\begin{proof}[Proof of Theorem \ref{thm1}]
We assume that $u_0$ is a nonzero solution of \cref{eq1}, then the $Pohozave$ identity is
\begin{eqnarray*}
	0&=&\frac{N-2}{2}\int_{\R }\left| \nabla u_0\right|^2 dx+\frac{N}{2}\int_{\R } \left| u_0\right| ^2dx\\&&+\frac{\lambda (N-s_1)}{p}\int_{\R }|x|^{-s_{1}}|u_0|^p d x-\frac{N-2}{2}\|u_0\|^{2^*(s_2)}_{2^*(s_2),s_2},
\end{eqnarray*}
and the $Nehari$ identity is
\begin{eqnarray*}
	\int_{\R }\left| \nabla u_0\right|^2 dx+\int_{\R } \left| u_0\right| ^2dx+\lambda\int_{\R }|x|^{-s_{1}}|u_0|^p d x-\|u_0\|^{2^*(s_2)}_{2^*(s_2),s_2}=0.
\end{eqnarray*}
Combining the above two equations, we can obtain
\begin{eqnarray*}
	\int_{\R } \left| u_0\right| ^2dx+\left( \frac{N-s_1}{p}-\frac{N-2}{2}\right) \lambda\int_{\R }|x|^{-s_{1}}|u_0|^p d x=0.
\end{eqnarray*}
Note that $$\frac{N-s_1}{p}-\frac{N-2}{2}\geq 0$$ from $p\leq 2^*(s_1)$. Then $u_0\equiv 0$, which contradicts our previous assumption that $u_0$ is a nonzero solution. Therefore, it is proven that \cref{eq1} has no nonzero solution.
\end{proof}

\begin{proof}[Proof of Theorem \ref{thm2}]
From \cref{prop:2,m^+}, we obtain $M^+\not=\phi$, and $0<m^+<+\infty$. It follows from Ekeland's variational principle \cite{Ekeland1974} that there exists a sequence $\{u_n\} \subset M^+$ with $u_n \geq 0$ such that
\begin{eqnarray*}
	I(u_n)\to m^+
\end{eqnarray*}
and
\begin{eqnarray}\label{IM0}
	(I|_M)'(u_n)\to 0
\end{eqnarray}
as $n\to\infty$.
By $u_n\in M^+$, one has
\begin{eqnarray*}
	\varphi(u_n)&=&\|u_n\|^2+\lambda\int_{\R }|x|^{-s_{1}}|u_n|^p d x-\int_{\R }|x|^{-s_{2}}|u_n|^q d x=0,\\
	\psi(u_n)&=&\|u_n\|^2+(p-1)\lambda\int_{\R }|x|^{-s_{1}}|u_n|^p d x-(q-1)\int_{\R }|x|^{-s_{2}}|u_n|^q d x<0,
\end{eqnarray*}
which implies that
\begin{eqnarray*}
	\lambda\int_{\R }|x|^{-s_{1}}|u_n|^p d x&=&\frac{q-2}{p-q}\|u_n\|^2+\frac{\psi(u_n)}{p-q},\\
	\int_{\R }|x|^{-s_{2}}|u_n|^q d x&=&\frac{p-2}{p-q}\|u_n\|^2+\frac{\psi(u_n)}{p-q}.
\end{eqnarray*}
The one obtains
\begin{eqnarray*}
	I(u_n)&=&\frac{1}{2}\|u_n\|^2+\frac{1}{p}\lambda\int_{\R }|x|^{-s_{1}}|u_n|^p d x-\frac{1}{q}\int_{\R }|x|^{-s_{2}}|u_n|^q d x\\
	&=&\frac{(p-2)(q-2)}{2pq}\|u_n\|^2-\frac{\psi(u_n)}{pq}\\
	&\geq&\frac{(p-2)(q-2)}{2pq}\|u_n\|^2.
\end{eqnarray*}
It follows that $\{u_n\}$ is bounded. Going if necessary to a subsequence, we can assume that
\begin{eqnarray*}
	&&u_n\rightharpoonup u_0 \ \ \ \ \  \ \ \ \ \ \  \ \ \  \ \text{in}\ H^1_r{(\R)},\\
	&&u_n\rightarrow u_0 \ \ \ \ \ \ \ \ \  \ \ \ \ \ \ \text{in}\ L^{q}( \R)\ \ (q\in (2,2^*)),\\
	&&u_n(x)\rightarrow u_0(x) \ \ \ \ \  \ \ a.e. \ \text{in}\  \R,\\
	&&\int_{\R }|x|^{-s_{1}}|u_n|^p d x\rightarrow\int_{\R }|x|^{-s_{1}}|u_0|^p d x,\\
	&&\int_{\R }|x|^{-s_{2}}|u_n|^q d x\rightarrow\int_{\R }|x|^{-s_{2}}|u_0|^q d x,\\
	&&\psi(u_n)\to -A
\end{eqnarray*}
as $n\rightarrow\infty$ for some $u_0\in \ H^1_r{(\R)}$ and some $A\geq 0$. Moreover, we have $u_0\geq 0$ and  $\lim\limits_{n\to\infty}\|u_n\|\geq\|u_0\|$.

Now we want to prove that $u_n \to u_0$ in $H^1(\R)$. It follows from Proposition 5.12 in \cite{Willem1996} that
\begin{align*}
	\left\|(\left.I\right|_M )^{\prime}\left(u_n\right)\right\|=\min _{\mu \in R}\left\|I^{\prime}\left(u_n\right)-\mu \varphi^{\prime}\left(u_n\right)\right\| .
\end{align*}
Then there exist $\mu_n \in R$ such that
\begin{align*}
	0 \leq\left\|I^{\prime}\left(u_n\right)-\mu_n \varphi^{\prime}\left(u_n\right)\right\| \leq\left\|(\left.I\right|_M )^{\prime}\left(u_n\right)\right\|+\frac{1}{n}
\end{align*}
for all $n \in N$. Therefore, by (\ref{IM0}), we have
\begin{align*}
	\left\|I^{\prime}\left(u_n\right)-\mu_n \varphi^{\prime}\left(u_n\right)\right\| \rightarrow 0
\end{align*}
as $n \rightarrow \infty$, which implies that
\begin{eqnarray*}
	\left|\left\langle I^{\prime}\left(u_n\right)-\mu_n \varphi^{\prime}\left(u_n\right),u_n\right\rangle\right|
	&=&|\mu_n| \left|\left\langle \varphi^{\prime}\left(u_n\right),u_n\right\rangle\right| \\
	&=&|\mu_n| \left| \varphi (u_n)+\psi (u_n)\right| \\
	&=&|\mu_n| \left| \psi (u_n)\right|.
\end{eqnarray*}
In the case $\psi(u_n)\to -A<0$, we can obtain  $\mu_n \to 0$ and $\|I'(u_n)\| \to 0$. Then we can know
\begin{eqnarray*}
	|\langle I'(u_n)-I'(u_0),\ u_n-u_0\rangle| &\leq&\|I'(u_n)\|\|u_n-u_0\|+|\langle I'(u_0),\ u_n-u_0\rangle|\\
	&\leq&C_1\|I'(u_n)\|+|\langle I'(u_0),\ u_n-u_0\rangle|\\
	&\rightarrow&0
\end{eqnarray*}
and
\begin{eqnarray*}
	&&\left|\int_{\R}|x|^{-s_{2}}(|u_n|^{q-2}u_n-|u_0|^{q-2}u_0)(u_n-u_0) d x\right|\\
	&\leq&\int_{\R}|x|^{-s_{2}}(|u_n|^{q-1}+|u_0|^{q-1})|u_n-u_0| d x\\
	&\leq&\int_{\R}|x|^{-s_{2}}|u_n|^{q-1}|u_n-u_0| d x+\int_{\R}|x|^{-s_{2}}|u_0|^{q-1}|u_n-u_0| d x\\
	&\leq&\left(\int_{\R}|x|^{-s_{2}}|u_n|^{q}d x\right)^\frac{q-1}{q}\left(\int_{\R}|x|^{-s_{2}}|u_n-u_0|^{q}d x\right)^\frac{1}{q} \\&&+\left(\int_{\R}|x|^{-s_{2}}|u_0|^{q}d x\right)^\frac{q-1}{q}\left(\int_{\R}|x|^{-s_{2}}|u_n-u_0|^{q}d x\right)^\frac{1}{q}\\
	&\leq&C_2\left(\int_{\R}|x|^{-s_{2}}|u_n-u_0|^{q}d x\right)^\frac{1}{q}\\
	&\rightarrow&0,
\end{eqnarray*}
where $C_1$ and $C_2$ are some positive constants.

By the definition of $I'$ we have
\begin{eqnarray*}
	\langle I'(u_n)-I'(u_0),\ u_n-u_0\rangle &=&(u_n-u_0,u_n-u_0)\\&&+\lambda\int_{\R}|x|^{-s_{1}}(|u_n|^{p-2}u_n-|u_0|^{p-2}u_0)(u_n-u_0) dx
	\\&&- \int_{\R}|x|^{-s_{2}}(|u_n|^{q-2}u_n-|u_0|^{q-2}u_0)(u_n-u_0) d x\\
	&\geq&\|u_n-u_0\|^2\\&&- \int_{\R}|x|^{-s_{2}}(|u_n|^{q-2}u_n-|u_0|^{q-2}u_0)(u_n-u_0) d x,
\end{eqnarray*}
Then one has $u_n \to u_0$ in $H^1_r(\R)$.

In the other case $\psi(u_n)\to -A=0$, let $$l=\lim\limits_{n\to\infty}\|u_n\|.$$ Then we have
\begin{eqnarray*}
	l^2+\lambda \int_{\R }|x|^{-s_{1}}|u_0|^p d x-\int_{\R }|x|^{-s_{2}}|u_0|^q d x&=&0,\\
	l^2+(p-1)\lambda \int_{\R }|x|^{-s_{1}}|u_0|^p d x-(q-1)\int_{\R }|x|^{-s_{2}}|u_0|^q d x&=&0,
\end{eqnarray*}
which implies that
\begin{eqnarray*}
	\lambda \int_{\R }|x|^{-s_{1}}|u_0|^p d x&=&\frac{q-2}{p-q}l^2,\\
	\int_{\R }|x|^{-s_{2}}|u_0|^q d x&=&\frac{p-2}{p-q}l^2.
\end{eqnarray*}
Note that $u_n \to u_0$ in $H^1_r(\R)$ if  $l=\|u_0\|$ and $u_n \rightharpoonup u_0$ in $H^1_r(\R)$. Assume that $l>\|u_0\|$, we have
\begin{eqnarray*}
	\varphi(u_0)&=&\|u_0\|^2+\lambda \int_{\R }|x|^{-s_{1}}|u_0|^p d x-\int_{\R }|x|^{-s_{2}}|u_0|^q d x\\
	&<&0.
\end{eqnarray*}
Hence, there exists $t_0\in (0, 1)$ such that $t_0 u_0\in M^+$ by \cref{prop:1}. Moreover, we obtain
\begin{eqnarray*}
	I(t_0u_0)
	&=&\frac{1}{2}\|u_0\|^2t_0^2+\frac{1}{p}\lambda \int_{\R }|x|^{-s_{1}}|u_0|^p d xt_0^p-\frac{1}{q}\int_{\R }|x|^{-s_{2}}|u_0|^q d xt_0^q-\frac{1}{2}\varphi(t_0u_0)\\
	&=&-\frac{p-2}{2p}\lambda\int_{\R }|x|^{-s_{1}}|u_0|^p d x t_0^p+\frac{q-2}{2q}\int_{\R }|x|^{-s_{2}}|u_0|^q d xt_0^q\\
	&=&h(t_0)	,
\end{eqnarray*}
where $$h(t)=-\frac{p-2}{2p}\frac{q-2}{p-q}l^2t^p+\frac{q-2}{2q}\frac{p-2}{p-q}l^2t^q.$$  Furthermore,  one has
\begin{eqnarray*}
	m^+&=&\frac{1}{2}l^2+\frac{1}{p}\lambda \int_{\R }|x|^{-s_{1}}|u_0|^p d x-\frac{1}{q}\int_{\R }|x|^{-s_{2}}|u_0|^q d x\\&=&\frac{1}{2}l^2+\frac{q-2}{p(p-q)}l^2-\frac{p-2}{q(p-q)}l^2\\&=&h(1).
\end{eqnarray*}
Note that
\begin{eqnarray*}
	h'(t)&=&-\frac{p-2}{2}\frac{q-2}{p-q}l^2t^{p-1}+\frac{q-2}{2}\frac{p-2}{p-q}l^2t^{q-1}\\&>&0
\end{eqnarray*} for $t\in (0,1)$ and $h'(1)=0$. Then we obtain  $h(1)>h(t_0)$, that is, $$m^+>I(t_0u_0),$$ which contradicts the fact $$m^+\leq I(t_0u_0)$$ by $t_0u_0\in M^+$ and the definition of $m^+$. Then one has  $l=\|u_0\|$, and $$u_n \to u_0 \quad \text{in}\ H^1_r(\R).$$

From $u_n \to u_0$ in $H^1_r(\R)$, we obtain $I(u_0)=m^+$ and $u_0 \in M^+ \cup M^0$. Assume that $u_0 \in M^0$. Then one has $$m^0\leq I(u_0)=m^+\leq m^0,$$ which implies that $$I(u_0)=m^0=m^+.$$ From Proposition \ref{m^+<m^0}, $m^+<m^0$, which contradicts $m^0=m^+$. So one has $u_0 \in M^+$. Hence,
$$\nabla(I|_{M^+})(u_0)=0.$$ Thus, there exists $\mu \in \RR$ such that
\begin{eqnarray*}
	\nabla I(u_0)=\nabla(I|_{M^+})(u_0)+\mu\nabla \varphi(u_0)=\mu\nabla \varphi(u_0),
\end{eqnarray*} which implies that  \begin{eqnarray*}
	0=\varphi(u_0)=(\nabla I(u_0), u_0)=\mu(\nabla\varphi(u_0),u_0)=\mu \psi(u_0).
\end{eqnarray*} Then  $\mu=0$ by $\psi(u_0)<0$, so $\nabla I(u_0)=0$, $I'(u_0)=0$ in $(H^1_r(\R))'$. By the principle of symmetric criticality due to Palais \cite{Palais1979}, we know $I'(u_0)=0$ in $(H^1(\R))'$. Then $u_0$ is a solution of \cref{eq1}. Note that $u_0\geq 0$ and $u_0\neq 0$. By the strong maximum principle \cite{Gilbarg2001} $u_0$ is a positive solution of \cref{eq1}. Moreover, it follows from Lemma 2.6 in \cite{Tang2025} that $u_0\in C(\R)$, which completes the proof.
\end{proof}

\subsection*{Data Availability} Data sharing not applicable to this article as no datasets were generated or analysed during the current study.

\subsection*{Declarations}
\textbf{Conflict of interest} The authors have no Conflict of interest to declare that are relevant to the content of this article.

\bibliography{ref}

\end{document}